\documentclass[12pt]{amsart}

\usepackage[english]{babel}
\usepackage[utf8x]{inputenc}
\usepackage[T1]{fontenc}
\usepackage{amsthm}
\usepackage{amsfonts}
\usepackage{amsmath}
\usepackage{amssymb}
\usepackage{enumitem}
\usepackage{indentfirst}

\usepackage{amsmath}
\usepackage{graphicx}
\usepackage[colorinlistoftodos]{todonotes}
\usepackage[colorlinks=true, allcolors=blue]{hyperref}

\newtheorem{lem}{Lemma}[section]
\newtheorem{definition}[lem]{Definition}
\newtheorem{theorem}[lem]{Theorem}
\newtheorem{prop}[lem]{Proposition}

\newtheorem{cor}[lem]{Corollary}
\newtheorem{prob}{Problem}

\makeatletter
\@namedef{subjclassname@2010}{%
  \textup{2010} Mathematics Subject Classification}
\makeatother

\DeclareMathOperator{\CL}{CL}
\DeclareMathOperator{\cl}{cl}

\DeclareMathOperator{\Pseudo}{Pseudo}

\begin{document}


\baselineskip=17pt



\title[SMALL MAD FAMILIES]{SMALL MAD FAMILIES WHOSE ISBELL-MRÓWKA SPACES ARE PSEUDOCOMPACT}

\author[V. O. Rodrigues]{Vinicius de Oliveira Rodrigues}
\address{Institute of Mathematics and Statistics\\ University of São Paulo\\
Rua do Matão, 1010, São Paulo, Brazil}
\email{vinior@ime.usp.br}

\author[A. H. Tomita]{Artur Hideyuki Tomita}
\address{Institute of Mathematics and Statistics\\ University of São Paulo\\
	Rua do Matão, 1010, São Paulo, Brazil}
\email{tomita@ime.usp.br}

\date{}

\begin{abstract}
	Given a countable transitive model $M$ for ZFC+CH, we prove that one can produce a maximal almost disjoint family in $M$ whose Vietoris Hyperspace of its Isbell-Mrówka space is pseudocompact on every Cohen extension of $M$. We also show that a classical example of $\omega_1$-sized maximal almost disjoint family obtained by a forcing iteration of length $\omega_1$ in a model of non CH  is such that the Vietoris Hyperspace of its Isbell-Mrówka space is pseudocompact. 

\end{abstract}



\maketitle

\section{Introduction}
The aim of this article is to give some examples of MAD families whose hyperspace of their Isbell-Mrówka spaces are pseudocompact by the use of forcing. Whether such spaces exist under the axioms of ZFC is a open problem, and after reviewing the literature, we concluded that the only positive consistent result is that under $\mathfrak p=\mathfrak c$, this is true for every MAD family. Nothing was known about the existence of a model where there are two MAD families, one whose Vietoris Hyperspace of its Isbell-Mrówka spaces is pseudocompact and the other is not. Our examples will come from MAD families of cardinality less than $\mathfrak c$, and such examples were not found in the literature.  

The study of topologies on collections of subsets of a given topological space started in the beginning of the past century. The first steps in this direction may be found in the first edition of \cite{hausdorff1927}, dating from 1915, where F. Hausdorff defines a metric on $\CL(X)=\{F\subseteq X: \emptyset \neq F \text{ is closed}\}$, where $X$ is a bounded metric space. 

The definition below is due to Leopold Vietoris and dates from the 1920's \cite{vietoris1922}.

\begin{definition}Given a topological space $X$, $\CL(X)$ is the collection of non-empty closed subsets of $X$.
	
	If $U_0, \dots U_n$ are open subsets of $X$, we define $\langle U_0, \dots, U_n\rangle=\{F \in \CL(X): F\subseteq \bigcup_{i\leq n}U_i, U_i\cap F\neq \emptyset ,  \forall i\leq n \}$. These sets forms a base for a topology, which is called the Vietoris Topology, and the space is called the Vietoris Hyperspace.
\end{definition}

In this article, we will simply refer to it as the hyperspace of $X$. We recall that this topology coincides with the topology generated by the Hausdorff metric whenever $X$ is a compact metric space.

We refer to E. Michael's work (\cite{michael1951}, 1951) for a collection of basic results regarding the Vietoris Topology. It is noticeable that the topological properties of $X$ have many relations to the topological properties of $\CL(X)$. In particular, there is a theorem due to Vietoris which states that $X$ is compact iff $\CL(X)$ is compact. This strong relationship induces natural questions about the existence of relationships between $X$ and $\CL(X)$ having properties that generalizes compactness, such as countable compactness, pseudocompactness\footnote{A topological space is pseudocompact if every continuous function from $X$ into $\mathbb R$ is bounded. There are other definitions for pseudocompactness that are only equivalent for Tychonoff spaces. $\CL(X)$ is not always Tychonoff even if $X$ is, however, if $X$ is Tychonoff, the standard definitions of pseudocompactness becomes equivalent for $\CL(X)$ \cite{ginsburg1977}} and the Lindelöf property.

In 1977, J. Ginsburg proved that, if $X$ is a Tychonoff and every power of $X$ is countably compact, then $\CL(X)$ is countably compact, and that if $\CL(X)$ is countably compact, then so is every finite power of $X$. Likewise, he has proven that if $\CL(X)$ is pseudocompact, then so is every finite power of $X$, and asked whether there is some relationship between the countable compactness (resp. pseudocompactness) of $X^\omega$ and of $\CL(X)$ \cite{ginsburg1977}.

While exploring Ginsburg's question, in 2004, J. Cao, T. Nogura and A. Tomita showed that under $\mathfrak p = \mathfrak c$, there exists a space $X$ such that $X^\kappa$ is countably compact for every $\kappa <2^{\mathfrak c}$ but its hyperspace is not countably compact. They also showed that if $X$ is a Tychonoff homogeneous space, if $\CL(X)$ is countably compact (resp. pseudocompact), then so is $X^\omega$ \cite{tomita2004}.

In 2007, M.  Hru\v{s}ák, F. Hernández-Hernández e I. Martínez-Ruiz explored this question on pseudocompactness by analyzing the Isbell-Mrówka spaces \cite{hrusak2007}.

\begin{definition}An almost disjoint family is an infinite collection of infinite subsets of $\omega$ whose intersections are pairwise finite. A MAD family (maximal almost disjoint) is a maximal almost disjoint family in the $\subseteq$ sense.
	
	Given an almost disjoint family $\mathcal A$, the Isbell-Mrówka space of $\mathcal A$, also called psi space of $\mathcal A$ is denoted by $\Psi(\mathcal A)$ is the set $\mathcal A\cup \omega$ topologized as follows: $\omega$ is open and discrete and $\{\{A\}\cup (A\setminus n): n \in \omega\}$ is a local basis for $A \in \mathcal A$.
	
\end{definition}
$\Psi(\mathcal A)$ is a Tychonoff, locally compact, zero dimensional, first countable noncompact topological space. A nice source of information about this class of spaces is the survey \cite{hrusak2014}. It is well known that MAD families of cardinality $\mathfrak c$ exists.

An almost disjoint family $\mathcal A$ is a MAD family iff $\Psi(\mathcal A)$ is pseudocompact iff $\Psi(\mathcal A)^\omega$ is pseudocompact. This way, by the theorems proved by Ginsburg, if $\mathcal A$ is an almost disjoint family such that $\CL(\Psi(\mathcal A))$ is pseudocompact, then $\mathcal A$ is MAD and $\Psi(\mathcal A)^\omega$ is pseudocompact. So, in the context of Isbell-Mrówka spaces,  Ginsburg's question becomes ``If $\mathcal A$ is a MAD family, is $\CL(\Psi(\mathcal A))$ pseudocompact?''. 

In \cite{hrusak2007}, M. Hru\v sák et al. proved that the answer to this question is independent of ZFC. Inspired by these techniques, they constructed a subspace $X$ of $\beta \omega$ such that $X^\omega$ is pseudocompact but $\CL(X)$ is not. In \cite{underreview}, we construct a subspace $X$ of $\beta \omega$ such that $X^\omega$ is countably compact but $\CL(X)$ is not pseudocompact. However, the following question, asked in \cite{hrusak2007} and also found in \cite{hrusak2014} remains open:

\begin{prob}[ZFC] Is there a MAD family $\mathcal A$ such that $\CL(\Psi(\mathcal A))$ is pseudocompact?
\end{prob}

The only previously known example of a pseudocompact hyperspace is under $\mathfrak p = \mathfrak c$, and our example is in models of $\mathfrak p= \omega_1< \mathfrak c$.

 For basic notation and theorems regarding this subject we refer to \cite{kunen2011set}, \cite{kunen1980set} and \cite{jech1978set}. We recall the definition of some of the cardinal characteristics of the continuum and some of their related concepts defined as in \cite{Blass}:

\begin{definition}We say that a collection $\mathcal A\subseteq [\omega]^{\omega}$ has the strong finite intersection property (SFIP) if the intersection of a nonempty finite subcollection of $\mathcal A$ has infinite intersection. A pseudointersection of $\mathcal A$ is an infinite set $B$ such that $B\subseteq^*A$ for every $A \in \mathcal A$.
	
	$\mathfrak p$, the pseudointersection number, is defined as the smallest cardinal $\kappa$ such that there exists a collection $\mathcal A\subseteq [\omega]^{\omega}$ such that $|\mathcal A|=\kappa$, $\mathcal A$ has the SFIP but has no pseudointersection.

	We say that a collection $\mathcal A\subseteq [\omega]^{\omega}$ is open dense if:
	
	\begin{itemize}
		\item $\forall A \in [\omega]^{\omega} \forall B \in \mathcal A\,(B\subseteq^* A\rightarrow B \in A)$, and
		\item $\forall A \in [\omega]^{\omega} \exists  B \in \mathcal A\, (B\subseteq^* A)$.
	\end{itemize}
	
	$\mathfrak h$, the distributivity number, is defined as the smallest size of a nonempty collection of open dense sets whose intersection is empty.
	
	Finally, $\mathfrak a$ is the smallest cardinality of a MAD family.
\end{definition}

It is known that $\omega_1\leq p\leq \mathfrak h\leq \mathfrak a\leq \mathfrak c$, that Martin's Axiom implies that $\mathfrak p=\mathfrak c$, and that all possible strictly inequalities are consistent. As stated above, the authors of \cite{hrusak2007} proved the following two theorems:

\begin{theorem}\label{hrusak}
	Under $\mathfrak p=\mathfrak c$, the hyperspace of the psi space of every mad family is pseudocompact.
\end{theorem}
\begin{theorem}\label{hrusak2}
	Under $\mathfrak h<\mathfrak c$, there exists a mad family whose hyperspace of its psi space is not pseudocompact.
\end{theorem}

Since $\mathfrak p=\mathfrak c$ implies $\mathfrak a=\mathfrak c$, theorem \ref{hrusak} only talks about mad families of cardinality $\mathfrak c$. The example constructed by the authors to prove theorem \ref{hrusak2} is also a mad family of cardinality $\mathfrak c$ since it is a mad family over the base set $2^{<\omega}$ whose every element is a chain or an antichain, and such a mad family must always have cardinality $\mathfrak c$. 

The authors mentioned that it was not clear from $\mathfrak h<\mathfrak c$ if there would exist some mad family whose Isbell-Mr\' owka space has its hyperspace pseudocompact. In this article, we give examples of models of $\mathfrak h < \mathfrak c$ in which  there exist small mad families whose hyperspaces of their psi spaces are pseudocompact.

\section{A criterion for pseudocompactness}

In this section, we aim to show that in order to verify that a space $\CL(\Psi(\mathcal A))$ is pseudocompact, we only need to verify that certain sequences have accumulation points. The criterion is the following proposition:

\begin{prop}\label{criterion}Suppose $\mathcal A$ is an almost disjoint family. Then $\CL(\Psi(\mathcal A))$ is pseudocompact if and only if for every sequence $C:\omega\rightarrow [\omega]^{<\omega}\setminus \{\emptyset\}$ of pairwise disjoint elements has an accumulation point in $\CL(\Psi(\mathcal A))$.
\end{prop}

This criterion appears in \cite{underreview} which is yet unpublished. For the sake of completeness, we sketch its proof here. First, we will need the following lemma that only talks about sequences of finite sets.

\begin{lem}\label{lemma1}Let $S$ be a sequence of finite sets. Then there exists $I \in [\omega]^\omega$ and sequences $U, D$ such that for every $n \in I$, $S(n)=U(n)\dot\cup D(n)$ and for every $n, m \in I$, if $n<m$ then $U(n)\subseteq U(m)$ and $D(n)\cap D(m)=\emptyset$.
\end{lem}
\begin{proof}Recursively, we define a strictly growing sequence $(x_n: n \in \omega)$ of natural numbers and a decreasing sequence $(J_n: n \in \omega)$ of infinite subsets of $\omega$ such that:
	
	\begin{enumerate}[label=(\arabic*)]
		\item $x_0=0$,
		\item $J_n\cap(x_n+1)=\emptyset$ for every $n \in \omega$,
		\item $x_{n+1}\in J_n$ for every $n \in \omega$, and
		\item $\forall n \in \omega\, \forall t \in C_n[\forall j \in J_n \, (t \in C_j )\vee \forall j \in J_n \, (t \notin C_j)]$.
	\end{enumerate}
	
	This is possible since each $C(n)$ is finite. Then let $I=\{x_n: n \in \omega\}$, $U(x_n)=\{t \in C(x_n): \forall j \in J_n(t \in C(j))\}$ and $D(x_n)=\{t \in C(x_n): \forall j \in J_n(t \notin C(j))\}$ for each $n \in \omega$.
\end{proof}

The criterion is a corollary of the following lemma:

\begin{lem}Let $X$ be a $T_1$ topological space, $D\subseteq X$ be a dense subset and $E=[D]^{<\omega}\setminus \{0\}$. Then:
	
	\begin{enumerate}[label=\roman*)]
		\item If every sequence in $D$ has an accumulation point in $X$, then $X$ is pseudocompact,
		\item if $D$ is open and discrete and $X$ is pseudocompact, then every sequence on $D$ has an accumulation point in $X$,
		\item $E$ is dense on $\CL(X)$,
		\item if $D$ is open and discrete, so is $E$,
		\item if every sequence of pairwise disjoint elements of $E$ has an accumulation point in $\CL(X)$, then the latter is pseudocompact.
	\end{enumerate}
\end{lem}

\begin{proof}i) If there exists an unbounded continuous function $f:X\rightarrow [0, \infty)$, choose $d_n \in D\cap f^{-1}[(n, \infty)]$. Then $(d_n: n \in \omega)$ has no accumulation point.
	
	ii) Suppose $(d_n: n \in \omega)$ is a sequence on $D$ with no accumulation point. Then $A=\{d_n: n \in \omega\}$ is infinite and clopen. So there exists a continuous function $f:X\rightarrow \mathbb R$ such that $f|A$ is unbounded and $f|(X\setminus A)$ is constant.
	
	iii) Since $X$ is $T_1$, $E\subseteq \CL(X)$. Given a nonempty basic open set $V=\langle U_0, \dots, U_n\rangle$, where $U_0, \dots, U_n$ are open in $X$, let $F\in V$. Then $F\subseteq \bigcup_{i\leq n}U_i$ and $F\cap U_i\neq \emptyset$ for each $i$. Let $x_i \in F\cap U_i$. It follows that $\{x_i: i \leq n\} \in V$.
	
	iv) Suppose $D$ is open and discrete. Suppose $F \in E$. Write $F=\{x_0, \dots, x_n\}$. Then $\{F\}=\langle \{x_0\}, \dots, \{x_n\} \rangle$.
	
	v) By $i)$ and $iii)$, it suffices to see that every sequence of elements of $E$ has an accumulation point in $\CL(X)$. So let $S$ be a sequence of elements of $E$. By lemma \ref{lemma1}, there exists $I \in [\omega]^\omega$ and sequences $U, D$ such that for every $n \in I$, $S(n)=U(n)\dot\cup D(n)$ and for every $m, n \in I$, if $m<n$ then $U(m)\subseteq U(n)$ and $D(m)\cap D(n)=\emptyset$. We break the proof into cases.
	
	\textbf{Case 1:} $U(n)=\emptyset$ for every $N \in I$. In this case, $D|I=S|I$ is a sequence of pairwise disjoint subsets of $E$, so by hypothesis it has an accumulation point.
	
	\textbf{Case 2:} $U(n)\neq \emptyset$ for some $N \in I$. Let $I'=I\setminus N$ and $F=\cl\bigcup_{n \in I'}U(n)$. Again, we break into two cases.
	
	\textbf{Case 2a:} There exists $M \in I'$ such that $D(n)=\emptyset$ for every $n\geq M$. In this case, let $I''=I\setminus M$. So $U|I''=S|I''$ is growing. The reader may verify that $S|I''$ converges to $F$.
	
	\textbf{Case 2b:} There exists $I'' \in [I']^{\omega}$ such that $D(n)\neq \emptyset$ for every $n \in I''$. By hypothesis, there exists an accumulation point $F'$ for $D|I''$. The reader may verify that $S|I''$ converges to $F\cup F'$.
\end{proof}

\begin{proof}[Proof of Proposition \ref{criterion}] Let $\mathcal A$ be an almost disjoint family. Notice that $X=\Psi(\mathcal A)$ is $T_1$ and $D=\omega\subseteq X$ is dense, so by letting $E=[D]^{<\omega}\setminus \{0\}$, v) of the previous lemma give us one side of the equivalence. For the converse, first notice that $\CL(X)$ is Hausdorff since $X$ is $T_1$ (\cite{michael1951}). By iii) and iv) of the previous lemma, $E$ is open and discrete, so by applying $ii)$ by swapping $D$ by $E$ and $X$ by $\CL(X)$, the converse follows.
\end{proof}

\section{Accumulation points for sequences}

In this section, we provide some criteria that implies that a sequence of pairwise disjoint finite nonempty subsets of $\omega$ has an accumulation point in $\CL(\Psi(\mathcal A))$. We start with a simple one:

\begin{lem}\label{lemma3}
	Let $\mathcal A$ be an almost disjoint family and let $X=\CL(\Psi(\mathcal A))$. Suppose $C=(C_n: n \in \omega)$ is a sequence of pairwise disjoint finite nonempty subsets of $\omega$. If there exists $F \in [\mathcal A]^{<\omega}$ such that $\{n \in \omega: C_n\subseteq \bigcup F\}$ is infinite, then $C$ has an accumulation point in $X$.
\end{lem}
\begin{proof}
	Set $I=\{n \in \omega: C_n\subseteq \bigcup F\}$ and enumerate $F$ as $\{A_0, \dots, A_k\}$.
	
	We show that there exists $J \in [I]^\omega$ such that for every $A\in F$, $\{n \in J: A\cap C_n\neq \emptyset\}$ is either $J$ or $\emptyset$. Recursively, we define a decreasing sequence $I_n \in [I]^\omega$ for $n\leq k+1$ as follows: Let $I_0=I$. After defining $I_n$ for $n<k+1$, let $I_{n+1}=\{m \in I_n: A_n\cap C_m\neq \emptyset  \}$ if this set is infinite. Otherwise, let $I_{n+1}=\{m \in I_n: A_n\cap C_m= \emptyset \}$. Finally, let $J=I_{k+1}$.

	Let $K=\{A \in F: \{n \in J: A\cap C_n\neq \emptyset\}=J\}$. $K$ is not empty, for if it was, then given $n \in J$, $C_n\cap \bigcup F=\emptyset$, but $n \in I$, so $C_n=\emptyset$, a contradiction. Also, notice that if $n \in J$ and $A \in F\setminus K$, then $C_n\cap A=\emptyset$. So $C_n\subseteq \bigcup K$.
	
	We claim that $(C_n: n \in J)$ converges to $K$: given open subsets $U_0, \dots, U_l$ of $\Psi(\mathcal A)$ such that $K\in \langle U_0,\dots,U_l\rangle$, there exists $M\in \omega$ such that for every $A \in K$, $A\setminus M\subseteq \bigcup_{i\leq l}U_i$ (because $K\subseteq \bigcup_{i\leq l}U_i$ and $K$ is finite). Also, for each $i\leq l$, there exists $N_i \in \omega$ and $A_i \in K$ such that $A_i\setminus N_i\subseteq U_i$. Finally, since the $C_n \,'$s are pairwise disjoint, there exists $m_0$ such that if $n\geq m_0$, then $C_n\cap \max\{N_0, \dots, N_l, M\}=\emptyset$. So if $n \in J\setminus m_0$, it follows that $\emptyset\neq C_n\cap A_i\setminus N_i\subseteq C_n\cap U_i$ for each $i\leq k$ and that $C_n\subseteq (\bigcup K)\setminus M\subseteq \bigcup_{i\leq l}U_i$.\\
\end{proof}

Later, we will construct almost disjoint families using forcing that satisfy the following criterion.

\begin{lem}\label{lemma4}
	Let $\mathcal A=\{A_\alpha: \alpha<\omega_1\}$ be a mad family. Suppose that there exists $\gamma<\omega_1$ and $I \in [\omega]^\omega$ such that:
	
	\begin{enumerate}[label=(\roman*)]
		\item For every $\xi<\gamma$, $\{n \in I: C_n\cap A_\xi\neq \emptyset\}$ is either finite or cofinite over $I$, and
		\item $\{\{n \in I: A_\xi\cap C_n\neq \emptyset\}: \gamma\leq \xi<\omega_1\}$ has the SFIP.
	\end{enumerate}
	
	Then, by letting $\mathcal A_0=\{A_\xi: \xi<\gamma\, \text{ and } \{n \in I:C_n\cap A_\xi\neq \emptyset\} \text{ is cofinite in } I\}$ and $\mathcal A_1=\mathcal A_0\cup \{A_\xi: \gamma\leq \xi< \omega_1\}$, $\mathcal A_1$ is an accumulation point of $(C_n: n \in I)$ in $\CL(\Psi(\mathcal A))$.
	
\end{lem}
\begin{proof}
	Let $J=\{\xi<\gamma: |I\setminus \{n \in I: C_n\cap A_\xi\neq \emptyset\}|<\omega\}$. Then $\mathcal A_0=\{A_\xi: \xi \in J\}$ and $\mathcal A_1=\{A_\xi: \xi \in J\cup [\gamma, \omega_1)\}$.
	Suppose $\langle U_0, \dots U_k\rangle$ is a neighborhood of $\mathcal A_1$, where $U_0$, \dots, $U_k$ are open subsets of $\Psi(\mathcal A)$. For each $i\leq k$, there exists $N_i \in \omega$ and $\xi_i \in J\cup [\gamma, \omega_1)$ such that $A_{\xi_i}\setminus N_i\subseteq U_i$. Let $K=\{\xi_i: i\leq k\}\cap [\gamma, \omega_1)$. By (2), $\{n \in I: \forall \xi \in K\, A_\xi\cap C_n\neq \emptyset\}$ is infinite.
	
	Since if $i\leq k$ and $\xi_i< \gamma$ then $\xi_i\in J$, it follows that $\{n \in I: A_{\xi_i}\cap C_n\neq \emptyset\}$ is cofinite on $I$, so $\{n \in I: \forall i\leq k\, A_{\xi_i}\cap C_n\neq \emptyset\}=\bigcap_{i\leq k}\{n \in I: A_{\xi_i}\cap C_n\neq \emptyset\}$ is infinite.
	
	Let $\tilde I=\bigcap_{i\leq k}\{n \in I: A_\xi\cap C_n\neq \emptyset\}\setminus \max\{N_i: i\leq k\}$. Notice that if $l\geq \max\{N_i: i\leq k\}$ and $l \in \tilde I$, then $\forall i\leq k$, $C_l\cap U_i\neq \emptyset$, so all that is left to see is that $\{n \in \tilde I: C_n\setminus \bigcup_{i\leq k} U_i\neq \emptyset\}$ is finite.
	
	Suppose by contradiction that it is infinite. Since the $C_n \,'$s are pairwise disjoint, $\bigcup_{n \in \tilde I} C_n\setminus \bigcup_{i\leq k}U_i$ is infinite, therefore there exists $\alpha$ such that $A_\alpha \cap \left(\bigcup_{n \in \tilde I} C_n\setminus \bigcup_{i\leq k}U_i\right)$ is infinite. If $\alpha\in\gamma\setminus J$, then, by (i), $\{n \in I: A_\alpha \cap C_n\neq \emptyset\}$ is finite, which implies $A_\alpha\cap \left(\bigcup_{n \in \tilde I} C_n\setminus \bigcup_{i\leq k}U_i\right)$ is finite, a contradiction. Else, $A_\alpha\in  \bigcup_{i\leq k}U_i$. Since the latter is open, it follows that $A_\alpha\subseteq^* \bigcup_{i\leq k}U_i$ so, again, $A_\alpha\cap \left(\bigcup_{n \in \tilde I} C_n\setminus \bigcup_{i\leq k}U_i\right)$ is finite, a contradiction.
\end{proof}

In order to apply the previous lemma, we need a special set $I \in [\omega]^\omega$. It will be useful to have a standard candidate for an $I$ that only depends on a given a countable almost disjoint family, an enumeration of it and a sequence of pairwise disjoint nonempty finite sets of naturals. First, we define a pseudointersection operator.

\begin{definition}
	Let $A=(a_n: n \in \omega)$ be a countable family of elements of $[\omega]^\omega$ with the SFIP. Let $\Pseudo(A)=\{\min(\bigcap_{k\leq n}a_k\setminus n): k\leq n\}$.
\end{definition}
Notice that $\Pseudo(A)$ is really a pseudointersection of $\{a_n: n \in \omega\}$ and that $\Pseudo(A)$ is absolute for transitive models of ZFC. Now we present the default candidate for an $I$.

\begin{definition}Given a infinite countable ordinal $\gamma$, a bijection $f:\omega\rightarrow \gamma$, a family $A=(A_\alpha: \alpha <\gamma)$ of distinct elements whose image is an almost disjoint family, $C=(C_n: n \in \omega)$ a sequence of nonempty finite pairwise disjoint subsets of $\omega$, one recursively defines:
	
	\begin{itemize}
		\item $I_0(A, C, f)=\omega$,
		\item $I_{m+1}(A, C, f)=\{n \in I_{m}: A_{f(m)}\cap C_n\neq \emptyset\}$, if  $\{n \in I_{m}: A_{f(m)}\cap C_n\neq \emptyset\}$ is infinite,
		\item $I_{m+1}(A, C, f)=I_m\setminus \{n \in I_{m}: A_{f(m)}\cap C_n\neq \emptyset\}$ otherwise,
		\item $I(A, C, f)=\Pseudo(I_{m}(A, C, F): m \in \omega)$.
	\end{itemize}

\end{definition}

The good thing about $I(A, C, f)$ is that it is absolute for transitive models of ZFC and that it always satisfies the first hypothesis of lemma \ref{lemma4}. We leave the proof to the reader.

\begin{lem} The operator $I$ defined above is absolute for transitive models of ZFC. Also, if $\gamma$ is an infinite countable ordinal, $f: \omega\rightarrow \gamma$ is a bijection, $A=(A_\alpha: \alpha <\gamma)$ is a family of distinct elements whose image is an almost disjoint family and $C=(C_n: n \in \omega)$ is a sequence of nonempty finite pairwise disjoint subsets of $\omega$, then $I(A, C, f) \in [\omega]^{\omega}$ and for every $\xi<\gamma$, either $\{n \in I: C_n\cap A_i\neq \emptyset\}$ or $\{n \in I: C_n\cap A_i= \emptyset\}$ is finite.
\end{lem}

\section{The First Example}

Let $M$ be a c.t.m. of $¬$CH. On $M$, there exists a c.c.c. iterated forcing notion of $\omega_1$ steps and finite supports $(\mathbb P_\alpha: \omega\leq\alpha\leq \omega_1^M) \in M$ such that if $G$ is $\mathbb P_{\omega_1}$-generic over $M$, then there exists a sequence $(A_\beta: \beta<\omega_1^M)$ (not in $M$) such that, by letting $\mathcal A=\{A_\alpha: \alpha<\omega_1^M\}$. 

\begin{enumerate}[label=(\arabic*)]
	\item $\mathcal A$ is an almost disjoint family,
	\item for every $\alpha \in [\omega, \omega_1^M]$, $(A_\beta: \beta<\alpha) \in M[G_\alpha]$ (so $\{A_n: n \in \omega\} \in M$ and $\mathcal A \in M[G])$,
	\item for every $\alpha \in [\omega, \omega_1^M]$ and for every $X \in [\omega]^\omega\cap M[G_\alpha]$, there exists $\beta\leq \alpha$ such that $|A_\beta\cap X|=\omega$.
\end{enumerate}

Notice that this implies that $\mathcal A$ is a mad family on $M[G]$, therefore $M[G]\vDash |\mathcal A|=\omega_1=\mathfrak h=\mathfrak a<\mathfrak c$. So, by Theorem \ref{hrusak}, working in $M[G]$ there exists a MAD family whose hyperspace of its Isbell-Mrówka space is not pseudocompact.

Such forcing notion may be found in \cite{halbesein} (pg 428, 105.), however, we quickly sketch how to construct it: first, notice that whenever $\mathcal A$ is a (infinite) countable almost disjoint family, one may consider $\mathbb Q_{\mathcal A}=[\omega]^{<\omega}\times [\mathcal A]^{<\omega}$ ordered by $(s, F)\leq (s', F')$ iff $s'\subseteq s, F'\subseteq F$ and $(s\setminus s')\cap \bigcup F'=\emptyset$. It is not difficult to verify that $\mathbb Q_\mathcal A$ has the c.c.c. and that if $G$ is $\mathbb Q_\mathcal A$-generic over $M$, then if $X=\bigcup\{s: (s, F)\in G\}$, $\mathcal A\cup\{X\}$ is an almost disjoint family and for all $A\in [\omega]^\omega\cap M$, there exists $B\in \mathcal A'$ such that $B\cap A$ is infinite. So by iterating it $\omega_1$ times with finite supports using names for these p.o.'s, where each step adds a new member to the almost disjoint family, we get the desired notion.

\begin{theorem}Let $M$ be a c.t.m. of $¬$CH. Then there exists a c.c.c. iterated forcing notion of $\omega_1$ steps and finite supports, $\mathbb P_{\omega_1}$, such that if $G$ is $\mathbb P_{\omega_1}$-generic over $M$, there exists an almost disjoint family $\mathcal A$ such that:
	
	$$M[G]\vDash |\mathcal A|=\omega_1, \,\mathcal A\text{ is MAD and } \CL(\Psi(\mathcal A)) \text{ is pseudocompact.}$$
	
	\begin{proof}Let $\mathbb P_{\omega_1}$ and $\mathcal A$ be as described above. Let $C=(C_n: n \in \omega) \in M[G]$ be a sequence on $[\omega]^\omega\setminus \{\emptyset\}$ of pairwise disjoint nonempty sets.\\
		
		\textbf{Case 1:} There exists $F \in [\omega_1^M]^{<\omega}\setminus \{\emptyset\}$ such that $I=\{n \in \omega: C_n\subseteq \bigcup_{\alpha \in F}A_\alpha\}$ is infinite. Working in $M[G]$, it follows by lemma \ref{lemma3}, that $(C_n: n \in \omega)$ has a convergent subsequence.\\
		
		\textbf{Case 2:} For every $F \in [\omega_1^M]^{<\omega}\setminus \{\emptyset\}$, the set $\{n \in \omega: C_n\setminus \bigcup_{\alpha \in F}A_\alpha\neq \emptyset\}$ is cofinite. In this case, since $\mathbb P_{\omega_1^M}$ satisfies the c.c.c., there exists an infinite $\mu<\omega_1^M$ such that $(C_n: n \in \omega)\in M[G_\mu]$.\\
		
		Let $f:\omega\rightarrow \mu$ be any bijection in $M$ and let $I=I((A_\beta: \beta<\mu), C, f)$.\\
		
		\textbf{Claim:} For every $K \in [[\mu, \omega_1^M[]^{<\omega}$, $\bigcap_{\xi \in K}\{n \in I: A_\xi\cap C_n\neq \emptyset\}$ is infinite.
		
		Proof of the claim: Write $K=\{\mu_1, \dots, \mu_k\}$, where $\mu_1<\dots<\mu_k$. Working in $M[G_{\mu_1}]$, write, $\mu_1=\bigcup_{m \in \omega}F_m$, where for each $m$, $F_m\subseteq \mu_1$ is finite and $F_m\subseteq F_{m+1}$. Since for every $m$ the set $\{n \in \omega: C_n\setminus \bigcup_{\alpha\in F_m}A_\alpha\neq \emptyset\}$ is cofinite, we may recursively choose a strictly growing sequence $n_m \in I$ and a sequence $k_m$ such that $k_m \in C_{n_m}\setminus \bigcup_{\xi \in F_m}A_\xi$. Since the $C_n$'s are pairwise disjoint, $X=\{k_m: m \in \omega\}$ is infinite and $\{X\}\cup\{A_\xi: \xi<\mu_1\}$ is an almost disjoint family, which implies that $X\cap A_{\mu_1}$ is infinite. Let $I_1=\{n \in I: A_{\mu_1}\cap C_n\neq \emptyset\} \in M[G_{\mu_1+1}]$. Since each $k_m$ belong to a different $C_n$, the set $I_1$ is infinite. Now we recursively repeat the argument for $n+1\leq k$ to get $I_{n+1}$ by using $I_n$ in the place of $I$, $\mu_{n+1}$ in the place of $\mu_1$. Notice that $I_{k}\subseteq \bigcap_{\xi \in K}\{n \in I: A_\xi\cap C_n\neq \emptyset\}$, which proves the claim.\\
		
Working on $M[G]$, it follows from Lemma \ref{lemma4}, that the sequence $(C_n: n \in \omega)$ has an accumulation point.
	\end{proof}
\end{theorem}

\begin{cor}\label{result}Con(ZFC)$\rightarrow$ Con(ZFC+ there exists a mad family $\mathcal A_0$ of cardinality $\omega_1$ and a MAD family $\mathcal A_1$ of cardinality $\mathfrak c>\omega_1$ such that $\CL(\Psi(\mathcal A_0))$ is pseudocompact but $\CL(\Psi(\mathcal A_1))$ is not.)
\end{cor}

\section{An example in the Cohen Model}
Before we start, we emphasize some differences between the previous example and the example in the Cohen Model. In the previous example, the MAD family is not in the ground model. This time, it will be. In the previous example, the ground model satisfies $¬$CH. This time, it will satisfy $CH$. So the two examples, despite using similar techniques, have some important differences.

For every $I, J$, let $\text{Fn}(I, J)=\bigcup_{a \in [I]^{<\omega}} J^{a}$. The Cohen forcing notion, defined as $\text{Fn}(\omega, 2)$ ordered by reverse inclusion, as usual. If $\kappa$ is an infinite cardinal, the forcing notion that adds $\kappa$ Cohen reals is $\mathbb C_\kappa=\text{Fn}(\kappa, 2)$. We refer \cite{kunen2011set} and \cite{kunen1980set} or \cite{jech1978set} for the basic properties of Cohen Forcing.

\begin{definition}
	Let $M$ be a c.t.m. and $\mathcal A \in M$ be an almost disjoint family. We say that $\mathcal A$ is indestructible for Cohen extensions if for every infinite cardinal $\kappa$ of $M$ and every $\mathbb C_\kappa$ generic filter $G$ over $M$, $(\mathcal A \text{ is MAD})^{M[G]}$. Since this is the only forcing we will mention from now on, we 
    will call it indestructible MAD family.
\end{definition}

We will modify the well~known construction of a indestructible MAD family in order to construct an example of a MAD family whose hyperspace of its Isbell-Mrówka space is pseudocompact in every Cohen extension. For the construction of a indestructible MAD family, we refer to \cite{kunen1980set}.
\begin{theorem}Let $M$ be a c.t.m. for ZFC+CH. Suppose:
	
	\begin{enumerate}[label=\alph*)]
		\item $((\dot C_\gamma, p_\gamma): \omega\leq \gamma<\omega_1^M) \in M$ be a listing of all pairs $(\dot C, p)$ such that:
		
		\begin{itemize}
			\item $\dot C$ is a $\mathbb C$-nice name for a subset of $\check{(\omega\times [\omega]^{<\omega})}$ in $M$,
			\item $p \in \mathbb C$,
			\item $p \Vdash \dot C:\omega\rightarrow \check{([\omega]^{<\omega}\setminus \{\emptyset\})} \text{ is a sequence of pairwise disjoint sets }.$
			
		\end{itemize}
		\item $(f_\gamma: \omega\leq\gamma<\omega_1) \in M$ is such that each $f_\gamma:\omega\rightarrow \gamma$ is bijective,
	\end{enumerate}
	
	Then there exists an indestructible almost disjoint family $\mathcal A=\{A_\alpha: \alpha<\omega_1^M\}$ in $M$ (enumerated in $M$) such that for every $\beta<\omega_1^M$, for every infinite $\gamma \leq\beta$, for every $F \in [\beta]^{<\omega}$,\\
    
if $\forall J \in [\beta]^{<\omega}(p_\gamma \Vdash |\{n \in I(\check{\mathcal A|\gamma}, \dot C_\gamma, \check f_\gamma): \forall \xi \in \check F\, (\dot C_\gamma(n)\cap A_\xi\neq \emptyset)$ and  $\dot C_\gamma(n)\setminus \bigcup_{\xi \in J}{A_\xi}\neq \emptyset\}|=\omega)$,\\

then $p_\gamma \Vdash |\{n \in I(\check{\mathcal A|\gamma}, \dot C_\gamma, \check f_\gamma): \forall \xi \in F\cup \{\check\beta\}\,(\dot C_\gamma(n)\cap A_\xi\neq \emptyset) \}|=\omega,$\\

\noindent where $\mathcal A|\gamma=(A_\xi: \xi<\gamma)$
	\begin{proof} Enumerate $\{(q, \tau): \tau \text{ is a nice name for a subset of } \check \omega, q\in \mathbb C, q\Vdash |\tau|=\omega\}$ as $\{(q_\alpha, \tau_\alpha ): \omega\leq \alpha<\omega_1\}$.

		Let $\{A_n: n \in \omega\}$ be an almost disjoint family in $M$. Given $\beta \in [\omega, \omega_1^M)$, suppose we have defined $(A_\xi: \xi <\beta) \in M$ such that:
		
		\begin{enumerate}[label=\alph*)]
			\item $\{A_\xi: \xi<\beta\}$ is an almost disjoint family, and
			\item for all $\beta'<\beta$, for every infinite $\gamma \leq\beta'$, for every $F \in [\beta']^{<\omega}$,\\
            
            if $ \forall J \in [\beta']^{<\omega}\,( p_\gamma \Vdash |\{n \in I(\check{\mathcal A|\gamma}, \dot C_\gamma, \check f_\gamma): \forall \xi \in \check F\, (\dot C_\gamma(n)\cap A_\xi\neq \emptyset) \text{ and }  \dot C_\gamma(n)\setminus \bigcup_{\xi \in J}{A_\xi}\neq \emptyset\}|=\omega)$,\\
			
			then $p_\gamma \Vdash |\{n \in I(\check{\mathcal A|\gamma}, \dot C_\gamma, \check f_\gamma): \forall \xi \in \check F\cup \{\check \beta'\}\,(\dot C_\gamma(n)\cap A_\xi\neq \emptyset) \}|=\omega,$
			
			\item if $\forall \xi<\beta\, q_\beta \Vdash |\tau_\beta\cap \check A_\eta|<\check \omega$ then $\forall n \in \omega \forall s\leq q_\beta \exists r\leq s\exists m\geq n(m \in A_\beta \text{ and } r\Vdash \check m \in \tau_\beta)$. 
		\end{enumerate}
		
		Notice that a) and c) implies that the almost disjoint family will be indestructible for Cohen extensions (see  \cite{kunen1980set}). We must define $A_\beta$.
		
		Working in $M$, suppose $\{(r, F, \gamma, l): l \in \omega, r\leq p_\gamma, F\in [\beta]^{<\omega}, \gamma<\beta', \forall J \in [\beta]^{<\omega}\,(p_\gamma \Vdash |\{n \in I(\check{\mathcal A|\gamma}, \dot C_\gamma, \check f_\gamma): \forall \xi \in \check F\, (\dot C_\gamma(n)\cap A_\xi\neq \emptyset) \text{ and } \dot C_\gamma(n)\setminus \bigcup_{\xi \in J}{A_\xi}\neq \emptyset\}|=\omega)$ is nonempty and enumerate it as $\{(r_m, F_m, \gamma_m, l_m): m \in \omega\}$.
		
		For every $m \in \omega$, there exists $s_m\leq r_m$, $n_m, k_m > l_m$ such that $s_m\Vdash \check n_m \in I(\check{\mathcal A|{\gamma_m}}, \dot C_{\gamma_m}, \check f_{\gamma_m}), \forall \xi \in \check F_m \,\dot C_{\gamma_m}(\check n_m)\cap \check A_\xi\neq \emptyset \text{ and } \check k_m\in \dot C_{\gamma_{m}}(n_m)\setminus \bigcup_{i\leq m} \check A_{\gamma_i}.$
		
		$k_m$ may be picked greater than $l_m$ since $r_m\leq p_{\gamma_m}\Vdash (\text{the }\dot C_{\gamma_m}(n)$'s are pairwise disjoint). Let $A^0_\beta=\{k_m: m \in \omega\}$. If the preceding set is empty, just let $A_\beta^0$ be an infinite subset of $\omega$ almost disjoint from every $A_\xi$ $(\xi <\beta)$. If we let $A_\beta=A_\beta^0$, clearly $(A_\xi: \xi<\beta+1)$ satisfies a) and b). As in \cite{kunen1980set}, working in M we construct $A_\beta^1$ almost disjoint from $A_\xi$ for every $\xi<\beta$ which makes c) hold. Let $A_\beta=A_\beta^0\cup A_\beta^1$ and we are done.

	\end{proof}
	
	\begin{theorem}Let $M$ be a c.t.m. for ZFC+CH. There exists an indestructible MAD family for Cohen extensions such that in every Cohen extension, $\CL(\Psi(\mathcal A))$ is pseudocompact.
	\end{theorem}
	
	\begin{proof}Let $((\dot C_\alpha, p_\alpha): \omega\leq \alpha<\omega_1^M)$, $(f_\gamma: \omega\leq \gamma<\omega_1)$ and $\mathcal A=\{A_\alpha: \alpha<\omega_1^M\}$ be as in the previous theorem. We claim that in every Cohen extension, $\CL(\Psi(\mathcal A))$ is pseudocompact. Suppose $\kappa$ is an infinite cardinal and that $G$ is $\mathbb C_k$-generic over $M$. Suppose by contradiction that, in $M[G]$,  $\CL(\Psi(\mathcal A))$ is not pseudocompact. Then, in $M[G]$, there exists a sequence of finite nonempty pairwise disjoint subsets of $\omega$, $C:\omega\rightarrow [\omega]^\omega$, with no accumulation point in $\CL(\Psi(\mathcal A))$. By Lemma \ref{lemma3}, for every $J \in [\omega_1^M]^{<\omega}$, $\{n \in \omega: C_n\subseteq \bigcup_{\alpha \in J} A_\alpha\}$ is finite.
		
		Let $S\subseteq \kappa$ be infinite countable such that $Q_0=\text{Fn}(S, 2)$ and $H_0=G\cap Q$ are such that $C \in M[H_0]$. Let $Q_1=\text{Fn}(\kappa\setminus I)\cap G$ and $H=G\cap Q_1$. Then $M[H_0][H_1]=M[G]$. Since $Q_0\approx \mathbb C$, there exists a generic filter $K$ over $\mathbb C$ such that $M[K]=M[H_0]$.

		There exists $\dot C \in M^{\mathbb C}$ such that $\dot C_K=C$ and such that $\dot C$ is a nice name for a subset of $\check{\omega\times [\omega]^{<\omega}}$. Now working in $M[K]$, there exists $p \in K$ such that:
		\begin{enumerate}
			\item $p \Vdash \dot C:\check \omega\rightarrow \check{[\omega]^{<\omega}\setminus \{\emptyset\}} \text{ is a sequence of pairwise disjoint sets}$, and
			\item $\forall J \in [\omega_1^M]^{<\omega}\, (p \Vdash \, |\{n \in \check \omega: \dot C_n\subseteq \bigcup_{\alpha \in \check J} A_\alpha\}|<\omega).$
		\end{enumerate}
		
		So, there exists $\gamma \in [\omega, \omega_1)$ such that $(\dot C, p)=(\dot C_\gamma, p_\gamma)$. Working on $M[G]$, we aim to get a contradiction by appling lemma \ref{lemma4} by letting $I$ be $I(\mathcal A|\gamma, C, f_\gamma)$. We already know that (i) holds. Since having the SFIP is absolute for transitive models of ZFC, we may verify (ii) holds on $M[K]$. So let $F\in [[\gamma, \omega_1)]^{<\omega}$ and write $P=\{\alpha_0, \dots, \alpha_l\}$ with $\alpha_0<\dots<\alpha_l$. For $i\leq l$, let $P_i=\{\alpha_0, \dots, \alpha_i\}$. We proceed by induction for $i\leq l$ to show that:
		
		$$p_\gamma \Vdash |\{n \in I(\check{\mathcal A|\gamma}, \dot C_\gamma, \check f_\gamma): \forall \xi \in \check P_i\, (\dot C_\gamma(n)\cap A_\xi\neq \emptyset)\}|=\omega.$$
		
		which will complete the proof.

        To see that it holds for $i=0$, let $\beta=\alpha_0$. Then $\forall J \in [\beta]^{<\omega}\,( p_\gamma \Vdash |\{n \in I(\check{\mathcal A|\gamma}, \dot C_\gamma, \check f_\gamma): \forall \xi \in \check \emptyset\, (\dot C_\gamma(n)\cap A_\xi\neq \emptyset) \text{ and } \dot C_\gamma(n)\setminus \bigcup_{\xi \in J}{A_\xi}\neq \emptyset\}|=\omega)$	is logically equivalent to $\forall J \in [\beta]^{<\omega}\,( p_\gamma \Vdash |\{n \in I(\check{\mathcal A|\gamma}, \dot C_\gamma, \check f_\gamma): \dot C_\gamma(n)\setminus \bigcup_{\xi \in J}{A_\xi}\neq \emptyset\}|=\omega)$ which holds, by (2). Therefore:
		
		$$p_\gamma \Vdash |\{n \in I(\check{\mathcal A|\gamma}, \dot C_\gamma, \check f_\gamma): \forall \xi \in \{\check \alpha_0\}\, (\dot C_\gamma(n)\cap A_\xi\neq \emptyset)\}|=\omega$$
		
		Now, suppose we have proved or claim for some $i<l$. We prove it for $i+1$. This time, let $\beta=\alpha_{i+1}$. We already know that $p_\gamma \Vdash |\{n \in I(\check{\mathcal A|\gamma}, \dot C_\gamma, \check f_\gamma): \forall \xi \in \check P_i\, (\dot C_\gamma(n)\cap A_\xi\neq \emptyset)\}|=\omega.$ Again, by 2., it follows that	$\forall J \in [\beta]^{<\omega}\,( p_\gamma \Vdash |\{n \in I(\check{\mathcal A|\gamma}, \dot C_\gamma, \check f_\gamma): \forall \xi \in \check P_i\, (\dot C_\gamma(n)\cap A_\xi\neq \emptyset) \text{ and } \dot C_\gamma(n)\setminus \bigcup_{\xi \in J}{A_\xi}\neq \emptyset\}|=\omega)$, which implies that:
		
		$$p_\gamma \Vdash |\{n \in I(\check{\mathcal A|\gamma}, \dot C_\gamma, \check f_\gamma): \forall \xi \in \check P_i\cup \{\check \alpha_{i+1}\}\, (\dot C_\gamma(n)\cap A_\xi\neq \emptyset)\}|=\omega,$$ completing the proof.

	\end{proof}
	
\end{theorem}

Since in every Cohen model satisfying $\neg$ CH, $\mathfrak h<\mathfrak c$, this example gives a different proof of Corollary \ref{result}.

\section{Conclusions}
We gave some examples of MAD families of cardinality $\omega_1$ whose hyperspaces of their psi spaces are pseudocompact, contributing towards Problem 1. No examples of cardinality lesser than $\mathfrak c$ were known. Both examples consists of models where there exists MAD families $\mathcal A_1$ and $\mathcal A_2$ such that $\CL(\Psi(\mathcal A_1))$ is pseudocompact and such that $\CL(\Psi(\mathcal A_2))$ is not, what was also not known. However, Problem 1 is still open. We ask another question in this direction:

\begin{prob}Is every hyperspace of a psi space of a MAD family of minimum cardinality pseudocompact? In particular, if the MAD family has cardinality $\omega_1$?
\end{prob}

Of course, if $\mathfrak a>\omega_1$, then the answer to the second question is trivially true.

\begin{prob}Is there a MAD family of cardinality $\mathfrak c$ in the Cohen model for which the  hyperspace of its psi is pseudocompact? What about in the iterated forcing model that adds an $\omega_1$ MAD family?
\end{prob}

\section{Acknowledgements}
The authors thanks FAPESP for the received support (process numbers	
2016/26216-8, regular research project, and 2017/15502-2, PhD Project).

\bibliographystyle{plain}
\normalsize
\baselineskip=17pt
\bibliography{sample}

\end{document}